%% file: main.tex
\documentclass[a4paper,UKenglish,cleveref, autoref, thm-restate]{oasics-v2021}

\renewcommand{\copyrightline}{}

\title{Speed-Aware Network Design: A Parametric Optimization Approach}

\author{Ugo Rosolia}{Amazon Science \& Tech, Luxembourg}{urosolia@amazon.lu}{https://orcid.org/0000-0002-1682-0551}{}

\author{Marc Bataillou Almagro}{Amazon Science \& Tech, Luxembourg}{maalmagr@amazon.fr}{}{}

\author{George Iosifidis}{Delft University of Technology, Delft}{G.Iosifidis@tudelft.nl}{George Iosifidis}{}

\author{Martin Gross}{Amazon Science \& Tech, Luxembourg}{grosmar@amazon.lu}{}{}

\author{Georgios Paschos}{Amazon Science \& Tech, Luxembourg}{paschosg@amazon.lu}{https://orcid.org/0000-0002-5922-1612}{}

\authorrunning{U. Rosolia et al.} 

\Copyright{Ugo Rosolia, Marc Bataillou Almagro, George Iosifidis, Martin Gross, and Georgios Paschos} 


\nolinenumbers 

\EventEditors{}
\EventNoEds{2}
\EventYear{2025}
\EventDate{September 18-19, 2025}
\EventLocation{Warsaw, Poland}
\EventLogo{}
\SeriesVolume{1}
\ArticleNo{1}

\newcommand{\speedFunc}{I_d}
\newcommand{\speedFuncApp}{\tilde{I}_d}

\newcommand{\cvx}{\textrm{Cvx}}
\usepackage{algorithmic}
\usepackage{algorithm}
\usepackage{booktabs}  

\begin{document}

\maketitle

\begin{abstract}
Network design problems have been studied from the 1950s, as they can be used in a wide range of real-world applications, e.g., design of communication and transportation networks. In classical network design problems, the objective is to minimize the cost of routing the demand flow through a graph. 
In this paper, we introduce a generalized version of such a problem, where the objective is to tradeoff routing costs and delivery speed;
we introduce the concept of \textit{speed-coverage}, which is defined as the number of unique items that can be sent to destinations in less than 1-day. 
Speed-coverage is a function of both the network design and the inventory stored at origin nodes, e.g., an item can be delivered in 1-day if it is in-stock at an origin that can reach a destination within 24 hours.  
Modeling inventory is inherently complex, since inventory coverage is described by an integer function with a large number of  points (exponential to the number of origin sites), each one to be evaluated using historical data. To bypass this complexity, we first leverage a parametric optimization approach, which converts the non-linear joint routing and speed-coverage optimization problem into an equivalent mixed-integer linear program.
Then, we propose a sampling strategy to avoid evaluating all the points of the speed-coverage function. The proposed method is evaluated on a series of numerical tests with representative scenarios and network sizes. We show that when considering the routing costs and monetary gains resulting from speed-coverage, our approach outperforms the baseline by 8.36\% on average.
\end{abstract}

\input{sec1-introduction.tex}
%
\input{sec3-model.tex}
\input{sec4-solution.tex}
\input{sec5-experiments.tex}
\input{sec6-conclusions.tex}



\bibliography{mybib.bib}

\end{document}

%% file: sec1-introduction.tex
\section{Introduction}\label{sec:introduction}



\subsection{Background \& Motivation}

Expedite delivery services are becoming increasingly important for e-commerce supply chains such as Amazon, Alibaba and Walmart
as they improve directly customer experience and can indirectly contribute to attaining key sustainability goals.
 Indeed, the option of expedite delivery increases naturally the range of items customers are willing to purchase from e-commerce platforms instead of visiting physical retail (brick-and-mortar) shops, and this has been found to reduce transportation-based carbon emissions in many scenarios 
 \cite{co2-ecommerce2020, co2-ecommerce2017}. These reasons contribute to a growing volume of research aiming to improve expedite delivery by means of  last-mile routing optimization \cite{klein2022dynamic}, vehicle dispatch scheduling \cite{dispatch-TransScienc2018}, or innovative crowd-shipping models \cite{carbone2017rise}, among others. 

\begin{figure*}[t]
		\centering
		\includegraphics[trim= 0mm 0mm 0mm 0mm, clip, width=\linewidth]{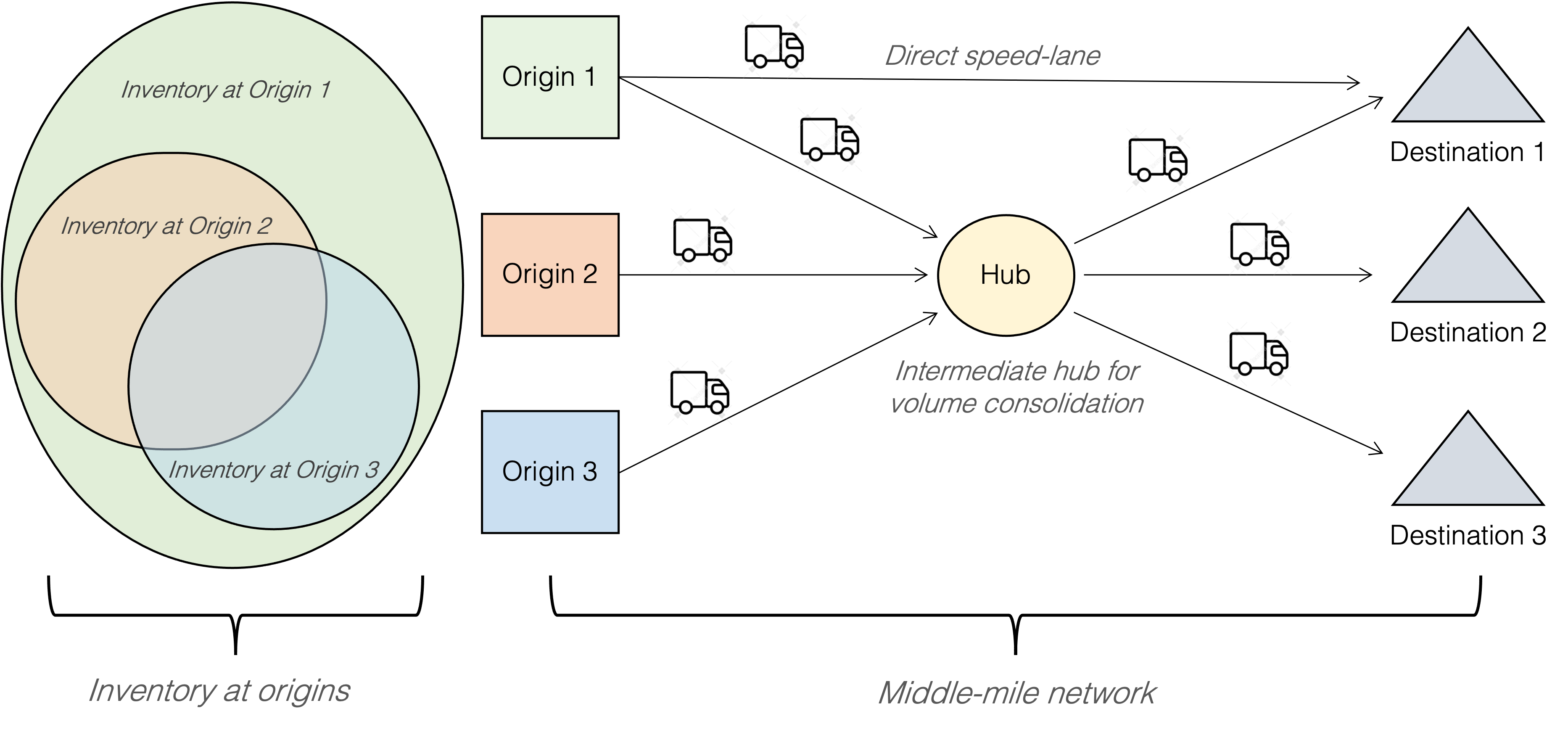}
		\captionsetup{width=.8\linewidth}
		\caption{\footnotesize{Middle-mile network design problem where destination nodes -- i.e., distribution centers -- are connected to  origin nodes -- i.e., warehouse where items are stored. Note that origin nodes have partially overlapping inventory, e.g., the inventory stored at Origin~2 and Origin~3 is also stored at Origin~1. Thus, connecting with speed-paths Origin~1 to destination nodes allows us to offer expedite delivery services for all items.}}\label{fig:high_level}
\end{figure*}

Nevertheless, the above works overlook the role of the middle-mile network connecting warehouses (origins) to distribution centers (destinations) in the efficacy and feasibility of such expedite delivery services. As the demand grows, items have to be stored at out-of-city warehouses, and therefore, the possibility of offering expedite delivery for these items is shaped by the network connections. This situation creates an unavoidable tension between cost and speed when designing the transportation network. To reduce costs we need to minimize the number of trucks used to ship orders to customers. Such a goal can be achieved by leveraging intermediate consolidation hubs between origin and destination nodes. Consider the example from Figure~\ref{fig:high_level}, where we have nine commodities — one for each origin-destination pair —  requiring one-third of truck’s capacity. By utilizing intermediate consolidation hubs, the network configuration from Figure~\ref{fig:high_level} requires only seven trucks to route packages. In contrast, serving these same nine commodities through \emph{speed-paths} -- i.e., paths directly connecting origins to destinations -- would require nine trucks, one for each origin-destination pair.  To tackle the tradeoff between volume consolidation and opening speed-paths, many studies formulate the middle-mile network design problem as a minimum-cost flow problem with maximum path-length constraints, or aim to balance transit times with path costs (see Sec. \ref{sec:lic_review}). 

Yet, a key factor that has been largely overlooked is the effect of the network connectivity on the share of inventory for which we can offer expedite deliver services.
In this paper, we propose to measure such an effect by evaluating the network's \textit{speed-coverage}, which we define as the number of unique items that can be delivered in 1-day. Speed-coverage is a function of ($a$) the network topology which determines the transit time from origins to destinations and ($b$) the inventory stored at origin nodes.
We recognize an inherent tradeoff: connecting an origin to a destination via a short route (speed-path)  increases the speed-coverage, but  may increase routing cost by making volume consolidation more difficult. Additionally, the benefit of speed-paths to speed-coverage exhibits diminishing returns. The more speed-paths we introduce, the less their differential impact to the overall speed-coverage objective. 
In this work, we formulate a problem that captures exactly these tradeoffs and then propose a technique to solve for the jointly optimal  routing costs and speed-coverage for the general case when origins have partially overlapping inventory. To the best of our knowledge, this paper is the first to introduce a mathematical formulation of these tradeoffs, and a scalable methodology for solving the optimization problem.

\input{sec2-related-work.tex}

\subsection{Paper Organization \& Notation} 

The rest of this paper is organized as follows. Section \ref{sec:prob_formulation} introduces the speed-aware minimum cost Multicommodity Capacitated Network Design problem under study. Section~\ref{sec:prop_approach} describes the proposed  reformulation based on a parametric interpolation, and the related approximation strategy; and Section~\ref{sec:experiments} presents a series of numerical experiments with representative scenarios and different network sizes. We conclude and present future works in Section~\ref{sec:conclusions}.

Throughout the paper we define the sets of (non-negative) reals and integers as ($\mathbb{R}^n_{+}$)$\mathbb{R}^n$ and ($\mathbb{Z}^n_{+}$)$\mathbb{Z}^n$, respectively. Vectors are denoted by boldface lowercase letters and sets by uppercase italic letters, e.g., $\boldsymbol{v} \in \mathbb{R}^n$ and $\mathcal{C}$. Given a set of vectors $\mathcal{C} = \{ \boldsymbol{v}_1, \boldsymbol{v}_2\}$, we denote the cardinality of $\mathcal{C}$ as $|\mathcal{C}|$ and its convex hull as $\cvx(\mathcal{C})$. Finally, we define the vectors of zeros $\boldsymbol{0}_{n} \in \mathbb{R}^n$, ones $\boldsymbol{1}_{n} \in \mathbb{R}^n$, and the unit base vector of zeros having one only in the $i$th component as $\boldsymbol{e}_n^i =[0, \ldots, 0, 1, 0, \ldots, 0]^\top \in \mathbb{R}^n$. 


%
%
%
%
%
%
%
%
%
%
%
%
%
%
%
%
%
%

%
%
%
%
%
%
%
%
%
%
%
%
%

%% file: sec2-related-work.tex
\subsection{Literature Review}\label{sec:lic_review}

As the importance of, and demand for, expedite delivery services grows, the design of speed-aware middle-mile networks becomes an increasing priority for service providers. Prior works on network design that cater for delivery speed include the hub-network design problem with time-definite delivery. In this setting, the objective is to decide hub locations together with routing paths, and speed deadlines are captured through path-eligibility constraints, see e.g., \cite{hub-location-campbell}. Similarly, \cite{hub-location-Transp21} optimizes the location of hubs, assignment of demand centers to hubs, and the vehicle routing. Expanding on these ideas, \cite{wu-snd-ts23} considers additionally hub capacity constraints; \cite{barnhart-express96} studies a single-hub overnight delivery system and optimizes routing subject to timing constraints; and \cite{yildiz-package-express} focuses on express air-service and decides which routes to operate with the company-owned cargo planes and how much capacity to purchase on commercial flights. Finally, \cite{dahan-lead-time} studies the middle-mile consolidation problem with delivery deadlines while accounting for consolidation delays. Similar models have been studied in the context of scheduled service network design that optimizes small (less-than-truckload) inter-city shipments with timing constraints for the delivery or intermediate hops, see \cite{hewitt-ssndp} and references therein. All the above works model the delivery time requirements through path-length constraints, e.g., a subset of origin-destination nodes are forced to be connected via paths that have transit time smaller than a predefined threshold. In contrast to this binary approach of enforcing time constraints, our work explicitly optimizes the tradeoff between speed and cost, providing a flexible framework for network design decisions.

Time-expanded graphs are one of the most common tools for capturing such timing restrictions but increase the problem's dimension and compound their solution, cf. \cite{konemann-time-expanded}. This modeling approach augments the dimension of the problem and therefore its complexity. For this reason, several approaches based on decomposition methods~\cite{yao2019admm, fragkos2021decomposition}, adaptive discretization techniques~\cite{boland2019perspectives}, and model condensation strategies~\cite{konemann-time-expanded} have been proposed. The key difference of our approach from these prior works is that we explicitly consider the overlap of inventory at origins to determine the number of unique items that are eligible for a 1-day delivery option. We do not assume to have an analytical model that captures how the delivery speed affects the inventory, but instead we provide a practical methodology for directly incorporating data-sets and look-up tables in the optimization problem. A key element of our strategy is a parametric linear model \cite{bank-parametric1982} which allows to express the inventory function in a compact and tractable form. Parametric optimization has been particularly successful for a range of applications, see \cite{pistikopoulos-multi-parametric}, but, to the best of the authors knowledge has not been employed for expedite delivery optimization in middle-mile networks.


\subsection{Methodology \& Contributions}\label{sec:met}


Jointly optimizing network design decisions and speed-coverage is a challenging problem. 
First of all, the network design problem, even without the speed-coverage objective, is already NP-hard to solve optimally when one has to ($a$) use unsplittable paths, ($b$) add integer link capacities (trucks), and ($c$) consider a large number of multi-hop path variables, \cite{routing-complexity, fortz-2017}. Secondly, the inventory effect of each new speed-path depends on which other origins are connected to a desination. In other words, speed-path assignment decisions are coupled across all origins serving a certain destination, but also across origins that use intersecting (at one or more edges) paths due to the edge capacities. Third, the inventory coverage function -- i.e., the function mapping a set of origins to the number of unique items stored at these nodes -- does not admit a convenient analytical expression. In fact, this function depends on the inventory overlap at origins, and in practice can be calculated using inventory datasets and look-up tables. This function format, unfortunately, does not facilitate --  actually prohibits -- its direct inclusion in the network optimization program, as it renders the problem highly nonlinear -- see equation~\eqref{eq:Id_function} from Section~\ref{sec:prob_formulation} for further details. 

Our method relies on multi-parametric optimization \cite{gal-multi-parametric, pistikopoulos-multi-parametric} to create a continuous (concave) approximation of the inventory function. This uses as input the available inventory data-points -- i.e., the number of unique items stored at a set of origin nodes -- and creates a continuous piecewise affine interpolation over the convex combination of the input data-points. We prove that this approximation is exact at the input data-points, and hence can serve as a meaningful proxy for maximizing this metric of interest. This, in turn, enables the inclusion of the inventory function in the network design problem, without inflating it with new discrete variables. The result of this new joint formulation is a \emph{speed-aware} middle-mile network, where the designer can tradeoff network costs with the number of unique items that can be delivered to customers in 1-day. Finally, we take an extra step to reduce the dimension of this problem through a sampling process, namely a low-complexity practical algorithm for selecting the approximation data-points used to approximate inventory overlap at origin nodes. 

In summary, the main contributions of this paper are as follows: $(i)$ We introduce the speed-aware network design problem of jointly optimizing network costs and speed-coverage, which we define as the share of unique items for which we can offer a 1-day delivery option. $(ii)$ We introduce a parametric-based modeling approach for enabling this joint optimization, overcoming the lack of a tractable analytical expressions for the number of unique items stored at a set of origin nodes. $(iii)$ We further propose a simple and practical sampling algorithm for reducing the size of the approximation problem, trading off interpolation accuracy in a systematic fashion. $(iv)$ The proposed joint model and approximation strategy are evaluated numerically in extensive tests on representatives scenarios.

%% file: sec3-model.tex
\section{Model and Problem Formulation}\label{sec:prob_formulation}

\subsection{Minimum cost MCND-U}
We first introduce the standard \emph{minimum cost Multicommodity Capacitated Network Design with Unsplittable demands} (MCND-U) problem, which aims to assign one path to each commodity and \emph{open} trucks on links to transport orders at minimum cost. 
Our network is modeled with a directed graph $\mathcal{G} = (\mathcal{V}, \mathcal{E})$, where $\mathcal{V}= \mathcal{O} \cup \mathcal{D} \cup \mathcal{H}$ is the set of nodes,  consisting of the set $\mathcal{O}$ of $n_O$ origins, the set $\mathcal{D}$ of $n_D$ destinations and the set $\mathcal{H}$ of $n_H$ other interim  nodes (or hubs), and $\mathcal{E}$ is the set of directed links. We are given a set $\mathcal{K}$ of $K$ commodities, where each commodity $k = (o_k, d_k)$ originates from an origin $o_k \in \mathcal{O}$ towards a destination $d_k\in \mathcal{D}$, and has volume $v_k\!\geq\! 0$.  We are also given a set of active network paths $\mathcal{P}$, where $\mathcal{P}_k$ is used to denote the active paths of commodity $k$ and $\mathcal{P}_{k\ni e}\subseteq \mathcal{P}_k$ the active commodity $k$ paths that traverse link $e \in \mathcal{E}$. Our goal is to assign exactly one path to each commodity 
and we do so  using the path selection vector $\textbf{x} =\left(x_p\in \{0, 1\}, p\in \mathcal{P}\right)$.\footnote{We do not need to index variables $x$ with commodities because each path $p$ can only be used by a single commodity, the one that corresponds to the origin-destination pair of the path.} Furthermore, the links have capacities that depend on our choice of opened trucks, where opening a truck costs $c_e$ on link $e\in \mathcal{E}$ and adds capacity $V_e\in  \mathbb{R}^n_{+}$. We use the truck deployment vector $\textbf{y} =\left(y_e\in \mathbf{Z}_+, e\in \mathcal{E}\right)$ to decide how many trucks are opened on each link. The minimum cost MCND-U problem is to assign paths to commodities with $\textbf{x}$ and open trucks with $\textbf{y} $ in order to transport the volumes at a minimum cost:
\begin{equation}
    \begin{aligned}\label{eq:prob_formulation_cost}
 \mathbb{P}_1: \qquad		\min_{\textbf{x}, \textbf{y}} \quad & \sum_{e \in \mathcal{E}} c_e y_e \\ 
		\text{subject to} \quad & \sum_{k \in \mathcal{K}}\sum_{p \in \mathcal{P}_{k\ni e}} v_k x_p \leq V_e y_e && \forall e \in \mathcal{E}, \\
		&\sum_{p \in \mathcal{P}_k} x_p = 1 && \forall k \in \mathcal{K}, \\		
		& x_p \in \{0,1\} &&\forall k \in \mathcal{K}, p \in \mathcal{P}_k, \\
		& y_e \in \mathbf{Z}_{+} &&\forall e \in \mathcal{E}.
\end{aligned}
\end{equation}
Problem $\mathbb{P}_1$ is the standard MCND-U problem studied in the literature, for instance see \cite{gendron1999multicommodity}, which is known to be \textit{NP}--hard \cite{routing-complexity, fortz-2017} -- See Section~\ref{sec:met} for details.  

\subsection{Modeling network speed}

In this work we augment the MCND-U problem with a novel  speed model, that captures the number of customer orders that can be delivered within a day, a service that is called \emph{Next Day Delivery} (NDD), see \cite{benidis2023middle}. Let us consider a destination node $d$, which serves a number of customers in a given area. Whenever a customer in that area makes an order, we say that the order is \emph{covered} (with NDD service) if  the ordered item is stored in any of the origins that are connected to $d$ with a \emph{short} path, where a path is said to be short if its total transit time is less than an input parameter; if path $p$ is short we will write $w_p=1$, else  it is long and we write $w_p=0$. Notice, we require two conditions for NDD: \emph{(i)} the ordered item is in the storage of some origin node, and \emph{(ii)} that origin node is connected to $d$ with a path that is short. First, let us study how we measure NDD coverage in a scenario where \emph{(ii)} is always satisfied, i.e., for now, we consider networks where all paths are short. We introduce the coverage function $I_d$ which counts the number of covered orders at a destination $d$.
It is typically impossible to derive the analytical form of $I_d$, but it is reasonable to estimate $I_d$ from available large datasets; using these datasets, one can compute an estimate of the number of covered unique NDD orders. 

\begin{figure}[h!]
	\centering
	\begin{subfigure}{0.45\textwidth}
		\centering
		\includegraphics[trim= 0mm 0mm 0mm 0mm, clip, width=1.0\linewidth]{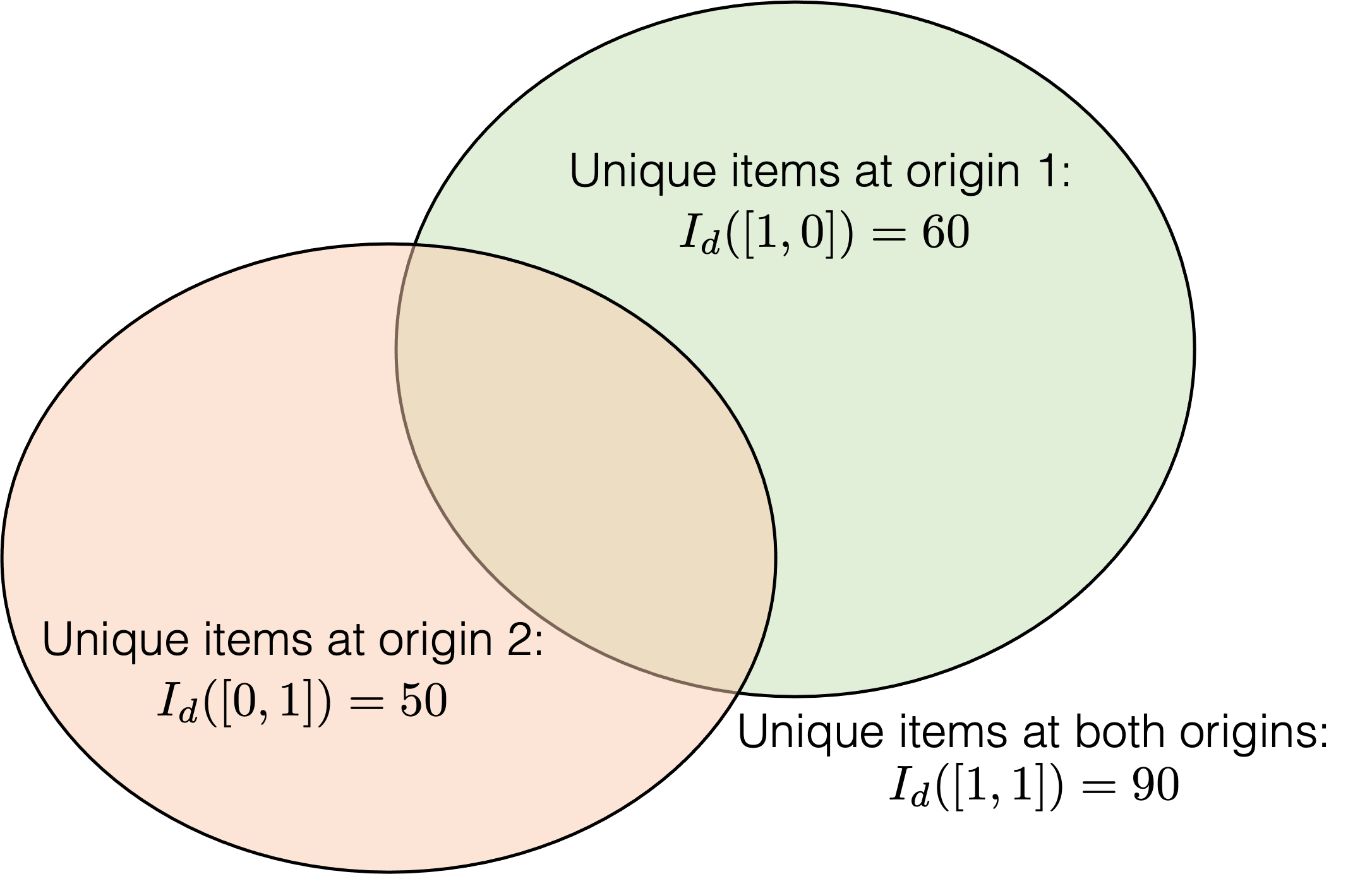}
		\captionsetup{width=1.0\linewidth}
	\end{subfigure}
	\caption{Venn diagram of two origin nodes storing items that are being ordered; although the sum of individually unique ordered items  stored in the two nodes is $50+60=110$, the actual total unique items  are only $90$ due to overlap between the two inventories.  
	}
	\label{fig:example_counting_gv}
\end{figure}

For example, see Figure~\ref{fig:example_counting_gv}. Assuming both origins of Figure~\ref{fig:example_counting_gv} are connected to $d$ with short paths, $I_d$ would evaluate to 90 items in this example, and we may imagine that in more complicated scenarios, the computation of $I_d$ would boil down into counting unique orders in large datasets of requests and inventories. 
The model, however, becomes more interesting when we study networks that include long paths, where condition \emph{(ii)} is not always satisfied. In such a case, depending on our path assignment decisions $\textbf{x}$, an origin-destination pair $(o,d)$ may not be  able to transport items eligible for NDD. For instance, if $(o,d)$ commodity is assigned a  long path, the contribution of $o$-inventory to $I_d$ should not be counted because although the inventory is available, its delivery takes more than one day. We should then adjust the coverage function $I_d$ to reflect this. A naive approach would be to adjust it to $I_d(\textbf{x})$ to denote the dependence on path assignment, 
but instead we use an alternative approach that   reduces significantly the dimensions, and hence the complexity of computing this function: 
we use the dependent speed variables 
\begin{equation}\label{eq:speed_variable}
    \textbf{z} =(z_{od}, (o,d)\in \mathcal{K}),
\end{equation} 
where $z_{od}=1$ if the path selected to connect $o$ to $d$ is short. Since we have only one assigned path per commodity (i.e., it holds $\sum_{p\in \mathcal{P}_k} x_p = 1$) it follows that
\[
z_{od}= \sum_{p\in \mathcal{P}_k} w_p x_p, ~~ \forall k=(o,d),
\]
where recall that $w_p =1$ if $p$ is short, and $0$ otherwise.
Using  variables $\textbf{z}$ we can write down the form of $I_d$ for the toy scenario of Figure~\ref{fig:example_counting_gv}:
\begin{equation}
\speedFunc(\textbf{z})= 60z_{1d}(1-z_{2d})+ 50z_{2d}(1-z_{1d})+90z_{1d}z_{2d}. \label{eq:example-Ud}
\end{equation}
More broadly, the function will have the form:
\begin{equation}\label{eq:Id_function}
\speedFunc(\textbf{z})= \sum_{s\in \mathcal{P}(\mathcal{O}_d)} \beta_s \prod_{i\in s} z_{id} \prod_{j \notin s}  (1-z_{jd}), 
\end{equation}
where $\mathcal{P}(.)$ is the power set, $\mathcal{O}_d$ is the set of origins that form a commodity with $d$, and $\beta_s$ is the unique item count of subset $s$ of origins -- we may obtain parameters $\beta_s$ by performing counting operations on our large datasets.
Note, that this function would need to have many more terms if we were to use path assignment variables $\textbf{x}$ (one additive term for each element in the powerset of active paths), hence we have effectively reduced the number of times we need to count parameters in large datasets  by an exponential factor. Also, observe that the form of \eqref{eq:Id_function} demonstrates well how deeply intertwined is the speed objective $\speedFunc(\textbf{z})$ with the cost optimization of  the MCND-U problem \eqref{eq:prob_formulation_cost}, both depending on our choice of $\textbf{x}$. Some paths may offer better consolidation by mixing volume in consolidation hubs, but because they are longer, they result in smaller NDD coverage. In the next subsection we augment the standard MCND-U problem with our newly introduced speed model.

\subsection{Speed-aware MCND-U}
The NDD coverage impacts customer satisfaction and shapes  long-term revenues. We model this effect with a \emph{customer conversion factor} $\gamma\!>\!0$; in practice, this factor can be estimated by analyzing customer behaviors via A/B testing. Therefore, our focus in this paper is to minimize the transportation costs and maximize long-term revenues from NDD coverage, which is formalized in the following  optimization program:
\begin{equation}\label{eq:prob_formulation}
\begin{aligned}
	\mathbb{P}_2:\qquad	\min_{\boldsymbol{x}, \boldsymbol{y}, \boldsymbol{z}} \quad & \sum_{e \in \mathcal{E}} c_e y_e - \gamma \sum_{d \in \mathcal{D}} \speedFunc(\boldsymbol{z})\\ 
		\text{subject to} \quad & \sum_{k \in \mathcal{K}}\sum_{p \in \mathcal{P}_{k\ni e}} v_k x_p \leq V_e y_e && \forall e \in \mathcal{E}, \\
		&\sum_{p \in \mathcal{P}_k} x_p = 1 && \forall k \in \mathcal{K}, \\
  & z_{od}=\sum_{p\in \mathcal{P}_k} w_p x_p, && \forall k=(o,d)\in\mathcal{K},\\
		& x_p \in \{0,1\} &&\forall k \in \mathcal{K}, \ p \in \mathcal{P}_k, \\
		& y_e \in \mathbf{Z}_{+} &&\forall e \in \mathcal{E}, \\
		& z_{od} \in\{0,1\} &&\forall o\in \mathcal{O}, d\in \mathcal{D}.
	\end{aligned}
\end{equation}

Problem $\mathbb{P}_2$ is extremely difficult to solve. Evidently, it is a generalization of the \textit{NP}--hard MCND-U, but a major additional  complexity factor is function  $I_d$. This coverage function is known to be submodular \cite{benidis2023middle}, i.e., the more short paths we assign (by switching $\boldsymbol{z}$ to $\boldsymbol{z'}$), the smaller  is the benefit $I_d(\boldsymbol{z'})-I_d(\boldsymbol{z})$ per added short path, due to overlapping inventory between origins. The domain of function $I_d(\boldsymbol{z})$  has $2^{n_O}$ points, each one of which requires a full computation on the large dataset. In fact, it is known that even finding the maximum point of such a function is an NP--hard problem \cite{nemhauser1978analysis, krause2014submodular}, let alone considering it inside a broader optimization problem as in $\mathbb{P}_2$. Last, observe that the objective of $\mathbb{P}_2$ is non-linear, see for example \eqref{eq:Id_function}. From previous studies in the space of maximum coverage problem \cite{ageev2004pipage} we know that the non-linear components of  \eqref{eq:Id_function} would make even the continuous relaxation of $\mathbb{P}_2$ \textit{NP}--hard.
For realistic e-commerce scenarios with multiple origins serving each destination, and a network with many destinations, solving $\mathbb{P}_2$ or even obtaining a lower bound via its continuous relaxation is intractable. 

In this paper, we overcome this challenge by leveraging multi-parametric programming \cite{gal-multi-parametric, pistikopoulos-multi-parametric}  to create an interpolation (or continuous extension) of integer functions $\speedFunc, d\in \mathcal{D}$, which in turn allows us to propose a solution methodology for the speed-aware MCND-U. 

%% file: sec4-solution.tex
\section{Solution Approach}\label{sec:prop_approach}

Our solution methodology is based on replacing the term $\speedFunc(\boldsymbol{z})$ in \eqref{eq:prob_formulation}  with a continuous extension, by interpolating integer points using parametric optimization. As we explain below, this interpolation is carefully designed to lead to a concave function, which eventually allows us to reformulate the speed-aware MCND-U as a Mixed-Integer Linear Programming (MILP). Finally, to reduce the dimensions and make the resulting problem tractable, we utilize an additional approximation technique to eliminate many of the integer points, and only consider a subset of them.

\begin{figure}[t]
	\centering
	\begin{subfigure}{0.465\textwidth}
		\centering
		\includegraphics[trim= 0mm 0mm 0mm 0mm, clip, width=1.0\linewidth]{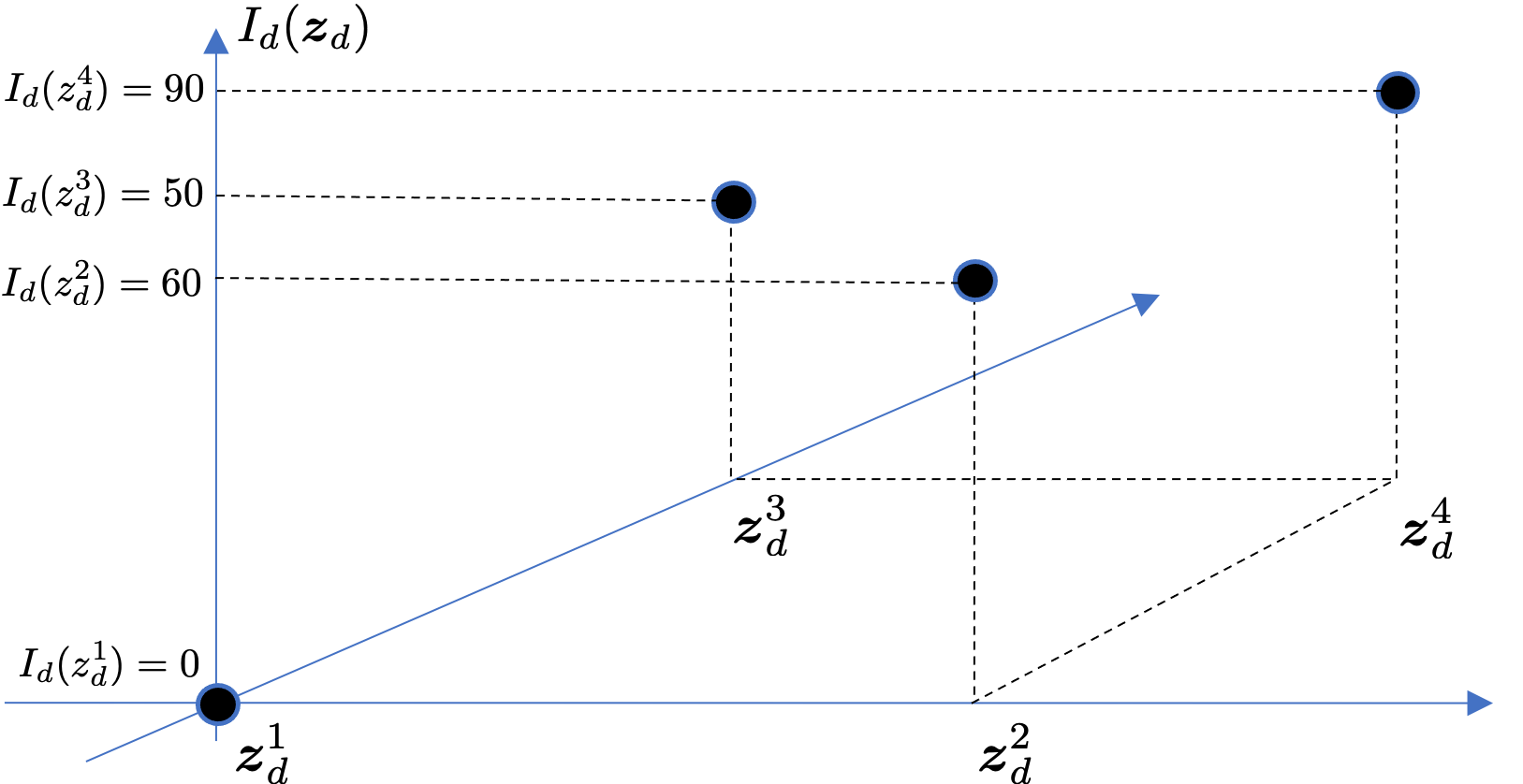}
		\captionsetup{width=.975\linewidth}
		\caption{\footnotesize{Dataset of speed-path assignments and unique items stored at origin nodes.}}\label{fig:interGV_construction}
	\end{subfigure}%
	\quad\quad~
	\begin{subfigure}{0.465\textwidth}
		\centering
		\includegraphics[trim= 0mm 0mm 0mm 0mm, clip, width=1.0\linewidth]{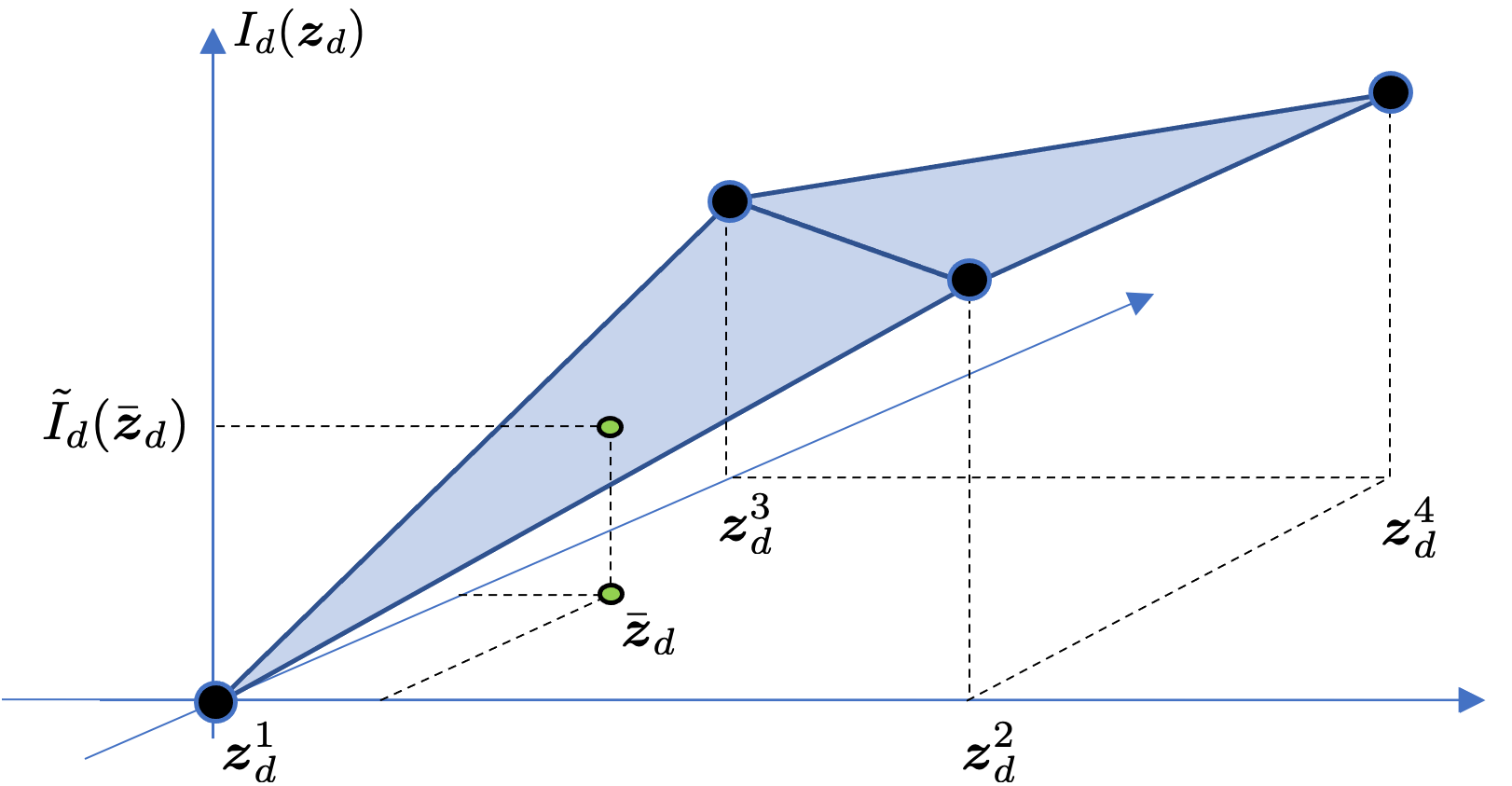}
		\captionsetup{width=.975\linewidth}
		\caption{\footnotesize{Piecewise linear interpolation of the speed-path assignement from Figure~\ref{fig:interGV_construction}.}}\label{fig:interGV_evaluation}
	\end{subfigure}
	\caption{Illustration of the approximation strategy for the example from Figure~\ref{fig:example_counting_gv}, where a destination $d$ is connected to Origin 1 and Origin 2.}
	\label{fig:2d_interpolation_example}
	\vspace{-0.4cm}
\end{figure}

\subsection{Parametric interpolation of unique inventory}\label{sec:param}


We will interpolate between  the integer points of function $I_d$, which are denoted with $\textbf{z}$ and introduced in~\eqref{eq:speed_variable}. For reasons that will become clear below when we present our approximation technique, we will introduce additional notation $\textbf{z}_d^i$ to enumerate the vector pointing to the $i^{\text{th}}$ integer point of $\speedFunc(\boldsymbol{z})$. Hence, 
\begin{equation}\label{eq:set_CPT}
	\mathcal{Z}_d = \left\{ \textbf{z}_d^1, \textbf{z}_d^2, \ldots, \textbf{z}_d^{n_\texttt{comb}} \right\},
\end{equation}
is another way to express the domain of $\speedFunc(\textbf{z})$ and we have $n_\texttt{comb} = 2^{n_O}$. See for example Figure~\ref{fig:interGV_construction}.

For any vector $\textbf{z}^i_d$, we can use inventory data to compute the number of unique items $I_d(\textbf{z}^i_d)$ for which we can offer NDD at $d \in \mathcal{D}$.  Then, for a vector $\textbf{z}_d$ in the convex hull of $\mathcal{Z}_d$, we can use the following optimization problem to interpolate $I_d(\textbf{z}^i_d)$:
\begin{equation}\label{eq:interpGV}
	\begin{aligned}
		\speedFuncApp(\textbf{z}_d) = \max_{\boldsymbol{\alpha_d} \in \mathbb{R}^{n_\texttt{comb}}_{+}} \quad &  \sum_{i= 1}^{n_\texttt{comb}} \alpha^i_d I_d(\textbf{z}_d^i) \\
		\text{subject to} \quad &  \sum_{i = 1}^{n_\texttt{comb}} \alpha_d^i \textbf{z}_d^i = \textbf{z}_d, \\
		&\sum_{i = 1}^{n_\texttt{comb}} \alpha_d^i = 1.  
	\end{aligned}
\end{equation}
Notice that problem~\eqref{eq:interpGV} is a linear parametric program and therefore $\speedFuncApp$ is a concave piecewise linear function~\cite[Chapter~6]{borrelli2017predictive}; we may think of $\speedFuncApp$ as an extension of $\speedFunc$, referred to in the literature as the concave closure~\cite{vondrak2011submodular}. We will  leverage the concavity of the extended function $\speedFuncApp$ to reformulate \eqref{eq:prob_formulation} as a MILP.  Note, that submodular functions may also be approximated by their convex closure, which is achieved by replacing the min operator with the max in~\eqref{eq:interpGV}. In this work, we have used the concave closure as this allows us to reformulate the speed-aware network design problem \eqref{eq:prob_formulation} as mixed-integer linear minimization problem, whereas using the convex closure would have resulted in a min-max mixed integer optimization problem.

We highlight that the interpolation from~\eqref{eq:interpGV} is exact at all integer points $\textbf{z}_d \in \mathcal{Z}_d$, since for each integer point $\textbf{z}_d^j \in \mathcal{Z}_d$ we must have that $\textbf{z}_d^j = \sum_{i=1}^{n_{\texttt{comb}}}\alpha_d^i \textbf{z}_d^i$, hence $\alpha_d^j = 1$ and $\speedFuncApp(\textbf{z}^j_d) = \alpha_d^j \speedFunc(\textbf{z}^j_d) =   \speedFunc(\textbf{z}_d^j)$. Note that $\speedFuncApp$ is exact at the integer points only when each $\textbf{z}_d^i \in \mathcal{Z}_d$ is an extreme point, i.e., it cannot be expressed as a convex combination of the vectors in $\{\mathcal{Z}_d \setminus \textbf{z}_d^i$\}. This condition holds in our case, and yields the following proposition:
\begin{proposition}\label{Prop1}
Let the function $\speedFuncApp$ be defined by the parametric program~\eqref{eq:interpGV}. For all integer points $\textbf{z}_d \in \mathcal{Z}_d$ we have that
\begin{equation*}
    \speedFunc(\boldsymbol{z}_d) = \speedFuncApp(\boldsymbol{z}_d).
\end{equation*}
\end{proposition}
\begin{proof}
Consider the optimization~\eqref{eq:interpGV} evaluated at $\textbf{z}_d^i$, i.e., $\tilde{I}(\textbf{z}_d^i)$. By definition we have that $\textbf{z}_d^i$ are vertices of a hypercube and hence $\textbf{z}_d^i \notin \cvx(\mathcal{Z}_d \setminus \textbf{z}_d^i)$ for all $\textbf{z}_d^i \in \mathcal{Z}_d$. Thus, we have that setting $\alpha_d^i = 1$ is the unique feasible solution to $\speedFuncApp(\textbf{z}^i_d)$. Thus, we conclude that $\speedFuncApp(\textbf{z}^i_d) = \alpha_d^i \speedFunc(\textbf{z}^i_d) =   \speedFunc(\textbf{z}_d)$.
\end{proof}

In Section~\ref{sec:paramopt} we combine the parametric extension of   $\speedFunc$ with the speed-aware MCND-U to obtain a MILP. 


\subsection{Parametric speed-aware MCND-U}\label{sec:paramopt}
Next, we reformulate the speed-aware MCND-U presented in \eqref{eq:prob_formulation} 
replacing the term $\speedFunc(\boldsymbol{z})$ with its parametric  interpolation $\speedFuncApp(\boldsymbol{z})$ from \eqref{eq:interpGV}: 
\begin{equation}\label{eq:prob_formulation_joint_optimization}
	\begin{aligned}
		\min_{\boldsymbol{x}, \boldsymbol{y}, \{\textbf{z}_d, \boldsymbol{\alpha}_d\}_{d\in\mathcal{D}}} \quad & \sum_{e \in \mathcal{E}} c_e y_e  - \gamma \sum_{d \in \mathcal{D}} \sum_{i= 1}^{n_\texttt{comb}} \alpha^i_d \speedFunc(\boldsymbol{z}^i_d)\\ 
		\text{subject to} \quad & \sum_{k \in \mathcal{K}}\sum_{p \in \mathcal{P}_{k\ni e}} v_k x_p \leq V_e y_e && \forall e \in \mathcal{E}, \\
		&\sum_{p \in \mathcal{P}_k} x_p = 1 && \forall k \in \mathcal{K}, \\
  & z_{od}=\sum_{p\in \mathcal{P}_k} w_p x_p, && \forall k=(o,d)\in\mathcal{K},\\
		& \sum_{i = 1}^{n_\texttt{comb}} \alpha_d^i \textbf{z}_d^i = \textbf{z}_d,\sum_{i = 1}^{n_\texttt{comb}} \alpha_d^i = 1 && \forall d \in \mathcal{D}, \\
		& x_p \in \{0,1\} &&\forall k \in \mathcal{K}, \forall p \in \mathcal{P}_k, \\
		& y_e \in \mathbf{Z}_{+} &&\forall e \in \mathcal{E}, \\
		& z_{od} \in \{0,1\} &&\forall o \in \mathcal{O}, d\in \mathcal{D}, \\		
		&\boldsymbol{\alpha}_d \in \mathbb{R}_{+}^{n_\texttt{comb}} && \forall d \in \mathcal{D}.
	\end{aligned}
\end{equation}


The above MILP is equivalent to the original problem~\eqref{eq:prob_formulation}, as the interpolation from~\eqref{eq:interpGV} is exact at  integer points, as discussed in Proposition~\ref{Prop1}. The main advantage of our reformulation is that it can be solved with off-the-self solvers. Unfortunately, for each destination $d \in \mathcal{D} $ there are up to $2^{n_\texttt{comb}}$ pre-computed vectors $\textbf{z}_d^i$ that make the reformulation intractable for real-world problems. Thus, in the next section, we propose an algorithm to select only a subset of assignments  to reduce the problem dimensionality. 

\subsection{Approximation Strategy}\label{sec:app}

Estimating the coverage function $\speedFunc(\textbf{z})$ at all integer points $\textbf{z}\in \mathcal{Z}_d$ from \eqref{eq:set_CPT} and then solving \eqref{eq:prob_formulation_joint_optimization} are both intractable due to the very large number of such points. In this section, we provide heuristics to select a subset of such vectors $\tilde{\mathcal{Z}}_d \subset \mathcal{Z}_d$, to simplify \eqref{eq:prob_formulation_joint_optimization} without losing accuracy on the resulting speed and cost tradeoff. Dropping vectors from $\mathcal{Z}_d$ means that certain paths will always be chosen to be long, and hence, our heuristics attempt  to drop the origins that are expected to have the smallest contributions to coverage. 

The first step is to rank all origins in terms of individual coverage achieved by taking the short path from this origin and long paths from all others,  denoted with $I_d(\boldsymbol{1}_o)$ for origin $o\in\mathcal{O}$. Calculating individual coverage is a cheap operation (only $n_O$ calculations per destination), and intuitively helps us prioritize origins with large inventory capabilities—notice, however, that due to inventory overlap, choosing the top origins in this regard does not guarantee maximal coverage. Below we use the notation  $\mathcal{T}_d(\kappa)\subseteq \mathcal{O}$ to denote  the top $\kappa$ origins with individual coverage. Note, that $\kappa$ is a user-defined hyper-parameter, which controls the complexity/performance tradeoff of our heuristic. 

Our strategy employs four sets of vectors, $\tilde{\mathcal{Z}}_d = \{\mathcal{H}_0, \mathcal{H}_1, \mathcal{H}_2, \mathcal{H}_3\}$ that contain a subset of all possible speed variable assignments. We use the terms ``short'' and ``long''  to refer to the speed variable assignments: when we say an origin has a ``short'' path to a destination ($z_{od}=1$), it means that origin $o$ can offer next-day delivery to destination $d$. Conversely, a ``long'' path ($z_{od}=0$) indicates that origin $o$ cannot provide next-day delivery to destination $d$. Based on this terminology we define the following sets:
\begin{enumerate}
\item $\mathcal{H}_0$: all short/long combinations of origins in $\mathcal{T}_d(\kappa)$ ($2^\kappa$ cases).
\item $\mathcal{H}_1$: all origins in $\mathcal{T}_d(\kappa)$ are short, and one of the remaining origins short ($n_o-\kappa$ cases)
\item $\mathcal{H}_2$: all origins in $\mathcal{T}_d(\kappa)$ are long, and one of the remaining origins short ($n_o-\kappa$ cases)	
\item $\mathcal{H}_3$: all origins in $\mathcal{T}_d(i), i = \kappa+1,\dots, n_o$ are short and the rest long ($n_o-\kappa$ cases)
\end{enumerate}
Note that our heuristic strategy has reduced significantly the amount of vectors for which we need to calculate coverage and add to \eqref{eq:prob_formulation_joint_optimization}, from $2^{n_O}$  down to $2^\kappa+3(n_O-\kappa)$, and we can use $\kappa$ to control how small this number is, trading off with the achieved accuracy of the optimization. We provide an illustrative example for $\kappa=2$ and $n_o=5$:
\begin{equation*}
    \begin{aligned}
        \mathcal{H}_0 = \{ &\tilde{\textbf{z}}_d^1 = \begin{bmatrix} 0 & 0 & 0 & 0 & 0 \end{bmatrix}^\top, \tilde{\textbf{z}}_d^2 = \begin{bmatrix} 1 & 0 & 0 & 0 & 0 \end{bmatrix}^\top,\\
        &\tilde{\textbf{z}}_d^3 =\begin{bmatrix} 0 & 1 & 0 & 0 & 0 \end{bmatrix}^\top, \tilde{\textbf{z}}_d^4 = \begin{bmatrix} 1 & 1 & 0 & 0 & 0 \end{bmatrix}^\top\}.
    \end{aligned}
\end{equation*}
\begin{equation*}
    \begin{aligned}
        \mathcal{H}_1 = \{ \tilde{\textbf{z}}_d^5 = \begin{bmatrix} 1 & 1 & 1 & 0 & 0 \end{bmatrix}^\top, \tilde{\textbf{z}}_d^6 = \begin{bmatrix} 1 & 1 & 0 & 1 & 0 \end{bmatrix}^\top, \tilde{\textbf{z}}_d^7 =\begin{bmatrix} 1 & 1 & 0 & 0 & 1 \end{bmatrix}^\top\}.
    \end{aligned}
\end{equation*}
\begin{equation*}
    \begin{aligned}
        \mathcal{H}_2 = \{ \tilde{\textbf{z}}_d^8 = \begin{bmatrix} 0 & 0 & 1 & 0 & 0 \end{bmatrix}^\top, \tilde{\textbf{z}}_d^9 = \begin{bmatrix} 0 & 0 & 0 & 1 & 0 \end{bmatrix}^\top, \tilde{\textbf{z}}_d^{10} =\begin{bmatrix} 0 & 0 & 0 & 0 & 1 \end{bmatrix}^\top\}.
    \end{aligned}
\end{equation*}
\begin{equation*}
    \begin{aligned}
        \mathcal{H}_3 = \{ \tilde{\textbf{z}}_d^{11} = \begin{bmatrix} 1 & 1 & 1 & 0 & 0 \end{bmatrix}^\top, \tilde{\textbf{z}}_d^{12} = \begin{bmatrix} 1 & 1 & 1 & 1 & 0 \end{bmatrix}^\top, \tilde{\textbf{z}}_d^{13} =\begin{bmatrix} 1 & 1 & 1 & 1 & 1 \end{bmatrix}^\top\}.
    \end{aligned}
\end{equation*}
\begin{algorithm}[t!]
	\caption{Heuristic to select subset $\tilde{\mathcal{Z}}_d$.}
	\label{algo:unique_item_app}
	\begin{algorithmic}[1]
		\REQUIRE $\kappa$.
		\FOR{$d \in \mathcal{D}$}
		\STATE Initialize $\tilde{\mathcal{Z}}_d = \emptyset$.
        \STATE Calculate $I_d(\boldsymbol{1}_o)$ for all $o$, rank them, and derive $\mathcal{T}_d(\kappa)$.
	
    \STATE Add $\mathcal{H}_0=\{\textbf{z}: z_{od} = 1 ~\forall o \in \mathcal{S}, z_{od} = 0 ~\forall o \notin \mathcal{S}, \forall \mathcal{S}\subseteq\mathcal{T}_d(\kappa)\}$.

    \STATE Add $\mathcal{H}_1= \left\{\textbf{z}: z_{od} = 1 ~\forall o \in \mathcal{T}_d(\kappa), z_{\tilde{o}d} = 1~\tilde{o}\notin \mathcal{T}_d(\kappa), z_{od} = 0 ~\forall o \notin \mathcal{T}_d(\kappa)\cup \{\tilde{o}\}\right\}$.

    \STATE Add $\mathcal{H}_2= \left\{\textbf{z}: z_{od} = 0 ~\forall o \in \mathcal{T}_d(\kappa), z_{\tilde{o}d} = 1~\tilde{o}\notin \mathcal{T}_d(\kappa), z_{od} = 0 ~\forall o \notin \mathcal{T}_d(\kappa)\cup \{\tilde{o}\}\right\}$.

    \STATE Add $\mathcal{H}_3=\{\textbf{z}:  z_{od} = 1 ~\forall o\! \in \! \mathcal{T}_d(i) \text{ and }  z_{od} = 0 ~\forall o \notin \mathcal{T}_d(i), i=\kappa+1, \ldots, n_O\}$.
    
		\ENDFOR
		\STATE \textbf{Return:} Set of vectors $\tilde{\mathcal{Z}}_d$ for all $d \in \mathcal{D}$.
	\end{algorithmic}
\end{algorithm}
Given $\tilde{\mathcal{Z}}_d$, we use the below program to approximate coverage:
\begin{equation}\label{eq:interpGV_app}
	\begin{aligned}
		\tilde{I}_d(\textbf{z}_d) = \max_{\boldsymbol{\alpha_d} } \quad &  \sum_{i} \alpha_d^i I_d(\tilde{\textbf{z}}_d^i) \\
		\text{subject to} \quad &  \sum_{i} \alpha_d^i \tilde{\textbf{z}}_d^i = \textbf{z}_d, \\
		&\sum_{i} \alpha_d^i = 1.
	\end{aligned}
\end{equation}
In the next section, we empirically analyze the effect of our approximation technique on run-time and solution quality.

%% file: sec5-experiments.tex
\section{Experiments}\label{sec:experiments}



We evaluate the proposed approach using a range of representative scenarios. First, we examine how optimizing jointly for speed and costs affects the network topology. Afterwards, we analyze the effect of the approximation from Algorithm~\ref{algo:unique_item_app} on the solution quality. To perform these analyses, we randomly generate networks with $n_O$ origins, $n_D$ destinations, and $5$ intermediate nodes that can be used to consolidate volume. For each origin-destination pair, the demand is randomly generated together with the travel times and transportation costs that are proportional to the traveled distance -- the location of all nodes is randomly selected from a uniform distribution. To serve one origin-destination pair, the optimizer can select either a direct path connecting the two nodes, or a path going through one of the five intermediate nodes. We assume that overall there are $n_{\texttt{it}} = 50 n_O$ unique items that we can offer to customers. For each origin $i$, we generate a random vector $v_i \in \{0, 1\}^{n_{\texttt{it}}}$, where each $j$th entry indicates if item $j$ is stored at origin $i$. If the travel time from an origin $o$ to a destination $d$ is less than ${\texttt{max\_tt}}=8$ hours, we assume that  we can offer NND for all items stored at origin $o \in \mathcal{O}$.

\begin{table}[t]
	\caption{Experimental results for four randomly generated network. For each network, we run the proposed method for $\gamma \in \{0, 0.1, 1\}$ to show the effect of revenues from NDD on the network topology.}
	\label{tab:run_comparison}
	\centering
	\begin{tabular}{rrrrrrrrc}
		\toprule
		$n_O$ & $n_D$ & \textbf{Costs} & \textbf{Avg. Items App.} & \textbf{Avg. Items} & $\gamma$ & Directs & Sol. Time & \%Gap \\
		\midrule
		10 & 10 & 689.1 & 0 & 342.8 & 0 & 20 & 58.4 & 0.06 \\ 
		10 & 10 & 698.5 & 383.7 & 383.7 & 0.1 & 26 & 10.9 & 0 \\
		10 & 10 & 731.9 & 404 & 404 & 1 & 30 & 2.6 & 0.02 \\ 
		\midrule
		10 & 20 & 1419.2 & 0 & 370.75 & - & 58 & 7200 & 1.85 \\
		10 & 20 & 1435.5 & 403.55 & 403.55 & 0.1 & 72 & 7200 & 2.51 \\
		10 & 20 & 1490.4 & 411.6 & 411.6 & 1 & 80 & 7200 & 0.19 \\ 
		\midrule
		20 & 10 & 1392.2 & 0 & 770.4 & 0 & 63 & 1196.1 & 0.1 \\ 
		20 & 10 & 1413.6 & 806.5 & 811.2 & 0.1 & 73 & 1799.2 & 0.1 \\ 
		20 & 10 & 1501.8 & 839.4 & 839.4 & 1 & 98 & 438.8 & 0.1 \\ 
		\midrule
		50 & 10 & 3528.2 & 2010 & 2032.2 & 0 & 149 & 7200 & 7.55 \\ 
		50 & 10 & 3489.8 & 2025.3 & 2037.3 & 0.1 & 146 & 7200 & 3.65 \\ 
		50 & 10 & 3653.4 & 2092.7 & 2092.7 & 1 & 178 & 7200 & 0.63 \\ 
		\midrule
		100 & 100 & 62204.2 & 0 & 3651.1 & 0 & 2254 & 7200 & 2.82 \\ 
		100 & 100 & 65258.6 & 4154.37 & 4172.14 & 0.1 & 3335 & 7200 & 8.43 \\
		100 & 100 & 66314.6 & 4211.57 & 4211.57 & 1 & 3383 & 7200 & 0.58 \\ 
		\bottomrule
	\end{tabular}
\end{table}


\begin{table}[b!]
	\centering
	\begin{tabular}{l|l|l|l|l|l|l|l|l}
		\midrule
		\textbf{$n_O$} & \textbf{$n_D$} & \textbf{Costs} & \textbf{App. Rev.} & \textbf{Rev.} & \textbf{Cost - (App. Rev.)} & \textbf{Cost - Rev.}  & \textbf{$\kappa$} & \textbf{\%Gap}  \\ \midrule
		\textcolor{blue}{10} & \textcolor{blue}{10} & \textcolor{blue}{689.1} & \textcolor{blue}{-} & \textcolor{blue}{342.8} & \textcolor{blue}{689.1} & \textcolor{blue}{346.3} & \textcolor{blue}{-} & \textcolor{blue}{0} \\ 
		10 & 10 & 698.5 & 383.7 & 383.7 & 314.8 & \textbf{314.8} & 1 & 0 \\ 
		10 & 10 & 698.5 & 383.7 & 383.7 & 314.8 & \textbf{314.8} & 5 & 0 \\ 
		10 & 10 & 698.5 & 383.7 & 383.7 & 314.8 & \textbf{314.8} & 10 & 0 \\ \midrule
		\textcolor{blue}{10} & \textcolor{blue}{20} & \textcolor{blue}{1419.2} & \textcolor{blue}{-} & \textcolor{blue}{370.7} & \textcolor{blue}{1419.2} & \textcolor{blue}{677.7} & \textcolor{blue}{-} & \textcolor{blue}{1.8} \\ 
		10 & 20 & 1459.7 & 405.9 & 407 & 647.9 & 645.7 & 1 & 3.4 \\ 
		10 & 20 & 1434.1 & 397.8 & 400.7 & 638.5 & 632.7 & 5 & 2.8 \\ 
		10 & 20 & 1435.5 & 403.5 & 403.5 & 628.4 & 628.4 & 10 & 2.5 \\ \midrule
		\textcolor{blue}{20} & \textcolor{blue}{10} & \textcolor{blue}{1392.2} & \textcolor{blue}{-} & \textcolor{blue}{770.4} & \textcolor{blue}{1392.2} & \textcolor{blue}{621.8} & \textcolor{blue}{-} & \textcolor{blue}{0.1} \\ 
		20 & 10 & 1422.9 & 807.3 & 815.7 & 615.6 & 607.2 & 1 & 0.1 \\ 
		20 & 10 & 1422.9 & 807.3 & 815.7 & 615.6 & 607.2 & 5 & 0.1 \\ 
		20 & 10 & 1413.6 & 806.5 & 811.2 & 607.1 & \textbf{602.4} & 10 & 0.1 \\ \midrule
		\textcolor{blue}{50} & \textcolor{blue}{10} & \textcolor{blue}{3375.8} & \textcolor{blue}{-} & \textcolor{blue}{1845.2} & \textcolor{blue}{3375.8} & \textcolor{blue}{1530.6} & \textcolor{blue}{-} & \textcolor{blue}{1.8} \\ 
		50 & 10 & 3528.2 & 2010 & 2032.2 & 1518.2 & 1496.0 & 1 & 7.5 \\ 
		50 & 10 & 3499.0 & 2031.5 & 2039.1 & 1467.5 & 1459.9 & 5 & 3.8 \\ 
		50 & 10 & 3489.8 & 2025.3 & 2037.3 & 1464.4 & \textbf{1452.4} & 10 & 3.6 \\ \midrule
		\textcolor{blue}{100} & \textcolor{blue}{100} & \textcolor{blue}{62204.2} & \textcolor{blue}{-} & \textcolor{blue}{3651.1} & \textcolor{blue}{62204.2} & \textcolor{blue}{25693.2} & \textcolor{blue}{-} & \textcolor{blue}{2.8} \\ 
		100 & 100 & 65911.4 & 4159.3 & 4175.7 & 24317.9 & 24154.2 & 1 & 11.2 \\ 
		100 & 100 & 65531.7 & 4156.8 & 4173.1 & 23963.2 & 23800.5 & 5 & 9.5 \\ 
		100 & 100 & 65258.6 & 4154.3 & 4172.1 & 23714.9 & \textbf{23537.2} & 10 & 8.4 \\ \midrule
	\end{tabular}
	\vspace{0.1cm}
	\caption{Experimental results for three randomly generated network. For each network, we run the proposed method for $\gamma = 0.1$ and $\kappa \in \{1, 5, 10\}$. As baseline we compute the cost optimal network. Thus, for the baseline (in blue) we do not have approximated revenues and a value for the parameter $\kappa$ that is not used to in the optimization.}\label{tab:run_theta_comparison}
\end{table}

\subsection{Trading-off transportation costs and speed}
In this section, we demonstrate through empirical experiments how the parameter $\gamma$ controls the trade-off between transportation costs and speed-coverage in the network design. Higher values of $\gamma$ result in networks that provide speed-paths for a broader range of unique items, while lower values prioritize cost minimization by encouraging consolidation through intermediate hubs.
Table~\ref{tab:run_comparison} shows results for five randomly generated networks. For each network, we solve three optimization problems: in the first problem, we optimize only for transportation costs -- i.e., we set $\gamma = 0$; in the second problem we set $\gamma = 0.1$; and in the third we use $\gamma = 1$. The solver terminates either when the gap is below $0.1\%$ or the solution time exceeds two hours. As expected, for larger values of $\gamma$ both transportation costs and the average number of unique items eligible for NDD increase, i.e., the optimizer decides to open more expensive connections to gain revenues from NDD. Note that the number of direct paths (not crossing any intermediate hub) increases with $\gamma$, as they have lower transit time and thus are more likely to offer NDD. In all experiments, we set the parameter $\kappa$ from Algorithm~\ref{algo:unique_item_app}  equal to $10$, meaning that the approximation from~\eqref{eq:interpGV_app} is exact for the unique items delivered by the top $10$ origins for each destination -- see Section~\ref{sec:res_theta_effect} for an empirical analysis on the effect of $\kappa$ on the solution. Indeed, we notice from Table~\ref{tab:run_comparison} that for $n_O = 10$ the approximate average number of unique items eligible for NDD (column \textbf{Avg. Items App.}) matches the exact number (column \textbf{Avg. Items}). On the other hand, for $n_O = 100$, the proposed strategy only computes a lower bound to the average number of unique items eligible for NDD at a destination $d \in \mathcal{D}$.
 
 \begin{table}[t]
 	\caption{Experimental results showing how the computational time changes as a function of the user defined parameter $\kappa \in \{1, 5, 10\}$.}
 	\label{tab:run_kappa_comparison}
 	\centering
 	\small
 	\begin{tabular}{rrrrrrrrc}
 		\toprule
 		$n_O$ & $n_D$ & \textbf{Cost-(App. Rev)} & \textbf{Cost-Rev} & Pre. Time & Sol Time & Gap(\%) & $\kappa$ & $n_\texttt{app}$ \\ 
 		\midrule
 		20 & 10 & 1392.2 & 621.8 & 0.08 & 1184.9 & 0.1 & - & 0 \\ 
 		20 & 10 & 615.6 & 607.2 & 0.08 & 577.8 & 0.1 & 1 & 59 \\ 
 		20 & 10 & 615.6 & 607.2 & 0.11 & 261.2 & 0.1 & 5 & 89 \\
 		20 & 10 & 607.1 & 602.4 & 1.59 & 3851.9 & 0.1 & 10 & 1081 \\ 
 		20 & 10 & 602.8 & 601.2 & 6.52 & 3146.9 & 0.1 & 12 & 4153 \\ 
 		20 & 10 & 599.0 & 599.0 & 36.0 & 6871.8 & 0.1 & 14 & 16441 \\ 
 		20 & 10 & 599.0 & 599.0 & 334.7 & 6978.5 & 0.1 & 16 & 65593 \\
 		\bottomrule
 	\end{tabular}
 \end{table}

\subsection{The effect of the proposed approximation}\label{sec:res_theta_effect}
In this section, we test how the parameter $\kappa$ from Algorithm~\ref{algo:unique_item_app} affects the solution quality.  Table~\ref{tab:run_theta_comparison} shows results for five different networks and parameter $\kappa \in \{1, 5, 10\}$. The table reports also results for the baseline cost-optimal network (in blue) that is obtained minimizing only the transportation costs, i.e., the optimizer does not approximate revenues from NDD and therefore the parameter $\kappa$ is not required to run the baseline. For all cases we report the total cost defined as the difference between the transportation costs and revenues from speed-paths\footnote{As in Problem~\ref{eq:prob_formulation}, revenues from speed-paths are computed by summing all unique items eligible for NDD at each destination and using the conversion factor $\gamma$.}  (column \textbf{Cost-Rev.}). For $\kappa = 1$, in Algorithm~\ref{algo:unique_item_app} we do not consider different combinations of origins offering NDD at a destination, but we simply sort origins by number of unique items eligible for NDD and we use the heuristic described in Section~\ref{sec:app}. This strategy allows us to prune the origins combinations used to construct the approximation from~\eqref{eq:interpGV_app}. On the other hand for $\kappa \in \{5, 10\}$, after sorting the origins by the number of unique items eligible for NDD, we consider all possible combination for the top $\kappa$ origins and for the remaining we leverage the heuristic from Algorithm~\ref{algo:unique_item_app}. As discussed in Section~\ref{sec:app}, this approximation allows us to reduce the number of constraints needed to build function~\eqref{eq:interpGV}. However when $\kappa$ is smaller than the number of origins $n_O$, the number of unique items computed using~\eqref{eq:interpGV_app} is only an approximation. This fact is shown in Table~\ref{tab:run_theta_comparison}, where we compare  the true revenues (column \textbf{Rev.}) and the approximated revenues (column \textbf{App. Rev.}) obtained by multiplying the number of unique items eligible for NDD by the coefficient $\gamma$. Finally, we notice that the total cost (column \textbf{Cost-Rev.}) defined as the difference between the transportation cost and the revenues from NDD decreases for larger value of $\kappa$. 
Note that for $\kappa = 1$, the approximation from~\eqref{eq:interpGV_app} is constructed using $n_{\texttt{app}} = 59$ data points, while for $\kappa = 10$ the approximation is constructed with $n_{\texttt{app}} = 1081$, i.e., for $\kappa = 10$  we increase by $94.5\%$ the number of data points used compared with the approximation with $\kappa = 1$. 
From the experiments in Table~\ref{tab:run_theta_comparison}, we notice that using $94.5\%$ more data points in the approximation from~\eqref{eq:interpGV_app} improves on average the network cost by $1.79\%$, i.e., as expected a more accurate approximation of revenues from NDD allows us to design a better network by trading off transportation costs and from NDD services.
It is also interesting to notice that even for $\kappa = 1$, our approach is able to reduce the overall cost -- i.e., the cost defined as the difference between transportation cost and revenues from speed-paths -- by 8.36\% compared to the baseline (in blue).

Finally, we investigate the effect of the parameter $\kappa$ on the computation time. Table~\ref{tab:run_kappa_comparison} shows both solver times and pre-processing times needed to compute the data points $n_\texttt{app}$ used to construct the approximation from~\eqref{eq:interpGV_app}. We notice that increasing $\kappa$ leads to higher pre-processing and solver time, as more data points are used for approximating the number of unique items eligible for NDD. However, we notice that increasing $\kappa$ from $10$ to $16$ results in only a $0.57\%$ cost improvement, while the computation cost increases by $1.9$x. This result highlights the effectiveness of the approximation from Algorithm~\ref{algo:unique_item_app}, where we ranked origins by the number of unique items eligible for NDD and considered all possible combination for a subset of them.

%% file: sec6-conclusions.tex
\section{Conclusion}\label{sec:conclusions}

In this paper, we consider the problem of jointly optimizing transportation costs and the share of inventory with a Next Day Delivery (NDD) option. Compared to previous methods, we account for overlapping inventory at origin nodes and how it affects the selection of speed-paths, i.e., origin-destination connections with a short travel time that enable NDD. The inventory is modeled with a coverage function, that requires the computation of a very large number of integer points. To tackle this complexity, we present an approach based on parametric optimization to construct a continuous extension of the inventory coverage function. Such an approximation requires solving a parametric linear program where the number of constraints increase exponentially with the number of origin nodes. To mitigate this issue, we present a sampling algorithm to tradeoff the accuracy of our approximation with computational complexity. The efficacy of the proposed approach is demonstrated on randomly generated networks, where we show that our strategy beats the baseline approach that computes a cost optimal network without optimizing for revenues from NDD.